\documentclass[a4,12pt]{article}%

\usepackage{graphicx}
\usepackage{graphics}
\newcounter{figurecounter}
\usepackage{amsmath}
\usepackage{amsfonts}
\usepackage{amssymb}
\usepackage{enumerate}
\newtheorem{theorem}{Theorem}[section]

\newtheorem{definition}[theorem]{Definition}

\newtheorem{example}[theorem]{Example}

\newtheorem{lemma}[theorem]{Lemma}

\newtheorem{proposition}[theorem]{Proposition}

\newenvironment{proof}[1][Proof]{\textbf{#1.} }{\ \rule{0.5em}{0.5em}}

\newcommand{\graph}{{\rm graph}}

\newcommand{\vertices}{{\rm vert}}
\newcommand{\modulo}{{\rm mod\ }}
\newcommand{\phull}{{\textrm{b-hull}}}

\newcommand{\dN}{{\mathbb N}}

\newcommand{\dR}{{\mathbb R}}
\newcommand{\dZ}{{\mathbb Z}}

\newcommand{\calG}{{\cal G}}

\newcommand{\calN}{{\cal N}}

\newcommand{\ep}{\varepsilon}

\begin{document}

\title{Monovex Sets%
\thanks{The first author also uses the spelling ``Buhovski'' for his family name. The work of L.~Buhovsky was partially supported by the ISF grant 1380/13, by the Alon Fellowship, and by the Raymond and Beverly Sackler Career Development Chair. The work of E.~Solan was partially supported by ISF grant 323/13.}}
\author{Lev Buhovsky%
\thanks{School of Mathematical Sciences, Tel Aviv University,
Tel Aviv 6997800, Israel.
levbuh@post.tau.ac.il.},
Eilon Solan%
\thanks{School of Mathematical Sciences, Tel Aviv University,
Tel Aviv 6997800, Israel. eilons@post.tau.ac.il.}, and Omri N. Solan%
\thanks{School of Mathematical Sciences, Tel Aviv University,
Tel Aviv 6997800, Israel. omrisola@post.tau.ac.il.}}

\maketitle

\begin{abstract}
A set $A$ in a finite dimensional Euclidean space is \emph{monovex}
if for every two points $x,y \in A$ there is a continuous path within the set that connects $x$ and $y$ and is monotone (nonincreasing or nondecreasing) in each coordinate.
We prove that every open monovex set as well as every closed monovex set is contractible,
and provide an example of a nonopen and nonclosed monovex set that is not contractible.
Our proofs reveal additional properties of monovex sets.
\end{abstract}


\section{Introduction}

A set $A$ in a finite dimensional Euclidean space is \emph{monovex}
if for every two points $x,y \in A$ there is a continuous path within the set that connects $x$ and $y$ and is monotone
(nonincreasing or nondecreasing) in each coordinate.
In particular, whether or not a set is monovex depends on the choice of basis for the space.

Monovex sets arise in the study of stochastic games (Solan, 2016),
where an extension of the Kakutani's fixed-point theorem (Kakutani, 1941) for set-valued functions
with a closed graph and nonempty monovex values is needed.%
\footnote{Kakutani's fixed-point theorem states that any set-valued function $F$
from a convex and compact subset of $\dR^n$ to itself with a closed graph and nonempty convex values has a fixed point.}
By Eilenberg-Montgomery fixed-point theorem (Eilenberg and Montgomery, 1946) any set-valued functions from a convex compact subset of $\dR^n$
to itself with a closed graph and nonempty contractible values has a fixed point.
Consequently, our goal is to study contractibility of monovex sets.

In this paper we prove that every open monovex set, as well as every closed monovex set, is contractible.
We also provide an example of a nonopen and nonclosed monovex set that is not contractible.

\section{Definition and Main Results}

The concept that this paper studies is monovex subsets of a finite dimensional Euclidean space.

\begin{definition}
A set $A \subseteq \dR^n$ is \emph{monovex} if for every $x,y \in
A$ there is a continuous path $\gamma : [0,1] \to A$ that
satisfies the following properties:
\begin{enumerate}
\item[(M1)]   $\gamma(0) = x$ and $\gamma(1) = y$.
\item[(M2)]   $\gamma_i : [0,1] \to \dR$ is a monotone function (nondecreasing or nonincreasing) for every $i\in \{1,2,\ldots,n\}$.
\end{enumerate}
A path $\gamma$ that satisfies Condition (M2) is called \emph{monotone}.
\end{definition}

The image of a monovex set under a diagonal affine transformation is monovex, yet
a rotation of a monovex set need not be monovex.
Every convex set is in particular monovex.
If $A$ is a monovex set, then so is the projection of $A$ onto any ``coordinate subspace'', that is,
a subspace spanned by a collection of elements of the standard basis of $\dR^n$.
Every monovex subset of $\dR$ is convex,
yet there are monovex subsets of $\dR^2$ that are not convex (see Figure \arabic{figurecounter}).

\bigskip
\centerline{\includegraphics{figure.1} \ \ \ \ \ \ \ \ \ \ \ \ \ \ \ \ \ \ \ \ \ \ \ \ \ \includegraphics{figure.2}}
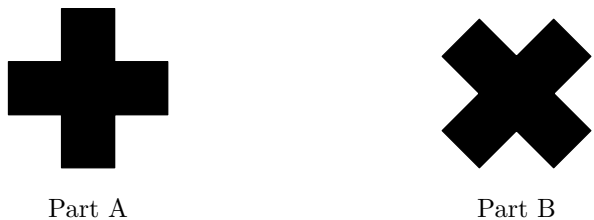
\centerline{Figure \arabic{figurecounter}: A monovex set (Part A) and a nonmonovex set (Part B) in the plane.}
\addtocounter{figurecounter}{1}
\bigskip

As the following example shows, monovex sets may be complex objects.
In particular, they need not be CW-complexes.

\begin{example}
\label{example:1}
Let $A \subset [0,1]^2$ be the following set (see Figure \arabic{figurecounter}):
\[ A = \{(0,0)\} \cup \left( \cup_{k=0}^\infty [\tfrac{1}{2^{k+1}},\tfrac{1}{2^k}]^2\right). \]
It is evident that the set $A$ is monovex, yet it is not a CW-complex.

\bigskip
\centerline{\includegraphics{figure.3}}
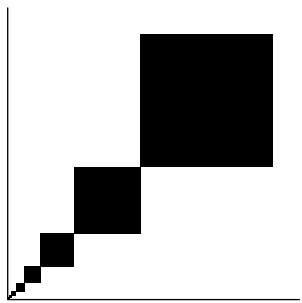
\centerline{Figure \arabic{figurecounter}: The monovex set $A$ in Example \ref{example:1}.}
\addtocounter{figurecounter}{1}
\end{example}

The Minkowski sum of two convex sets is a convex set.
This property is not shared by monovex sets.
In fact, as the following example shows,
the Minkowski sum of a monovex set and a convex set need not be a monovex set.
In Lemma \ref{lemma:0} below we will prove that the Minkowski sum of a monovex set in $\dR^n$ and an $n$-dimensional box whose faces are parallel to the axes is a monovex set.

\begin{example}
\label{example:3}
Let $A$ be the union of the two line segments $[(0,0,0),(0,1,1)]$ and $[(0,1,1),(1,1,2)]$, which is monovex.
Let $B$ be the line segment $[(0,0,0),(-1,-1,2)]$.
The intersection of the set $A+B := \{ a+b \colon a \in A, b \in B\}$ and the line $\{x \in \dR^3 \colon x_1=x_2=0\}$
is the two points $(0,0,0)$ and $(0,0,4)$.
Indeed, all points $b \in B$ satisfy $b_1= b_2$,
while the only points $a \in A$ that satisfy $a_1=a_2$ are $(0,0,0)$ and $(1,1,2)$.
Hence a point $a+b \in A + B$ is on the line $\{x \in \dR^3 \colon x_1=x_2=0\}$ if and only if
$a=b=(0,0,0)$ or $a=-b=(1,1,2)$.

Since the intersection of $A+B$ and the line $\{x \in \dR^3 \colon x_1=x_2=0\}$ contains two points,
there is no monotone path that connects these points and lies in $A+B$, and therefore the set $A+B$ is not monovex.
\end{example}

The Minkowski sum of the sets in Example \ref{example:3} is contractible.
As the following example shows, the Minkowski sum of a monovex set and a convex set can be homotopy equivalent to the circle $S^1$.

\begin{example}
\label{example:4}
Let $A$ be the union of the three line segments $[(0,0,0),(1,0,0)]$, $[(1,0,0),(1,1,0)]$, and $[(1,1,0),(1,1,1)]$, which is monovex.
Let $B = \{(x,x,x) \colon x \in \dR\}$, which is convex.
Denote by $C$ the triangle in $\dR^3$ whose vertices are $(0,0,0)$, $(\tfrac{2}{3},-\tfrac{1}{3},-\tfrac{1}{3})$, and $(\tfrac{1}{3},\tfrac{1}{3},-\tfrac{2}{3})$.
The Minkowski sum of $A$ and $B$ is $A+B = C+B$, which is homotopy equivalent to the circle $S^1$.
\end{example}

As mentioned in the introduction, our goal is to study whether monovexity implies contractibility.
It is a little technical but not difficult to show that every monovex subset of $\dR^2$ is contractible.
As the following example shows, not every three-dimensional monovex set is contractible.

\begin{example}
\label{example:2}
Let $A \subset [-1,1]^3$ be the set of all points that have at least one negative coordinate and at least one nonnegative coordinate.
The reader can verify that the set $A$ is monovex.
The set $A$ is disjoint of the line $\{(x,x,x) \colon x \in \dR\}$,
and it contains the loop $\gamma$ that is depicted in Figure \arabic{figurecounter} and is not contractible in $\dR^3 \setminus \{(x,x,x) \colon x \in \dR\}$.
In particular, the set $A$ is not contractible.
In fact, one can show that the set $A$ is homotopy equivalent to the circle $S^1$.

\bigskip
\centerline{\includegraphics{figure.5}}
\centerline{Figure \arabic{figurecounter}: The path $\gamma$ in Example \ref{example:2} (the dark curve).}
\addtocounter{figurecounter}{1}
\end{example}

We next observe that every open monovex set is contractible.

\begin{theorem}
\label{theorem:open}
Every open monovex subset of $\dR^n$ is contractible.
\end{theorem}

\begin{proof}
The proof is by induction on $n$.
For $n=1$, an open monovex set is an open interval, hence contractible.
Assume now that $A$ is an open monovex subset of $\dR^n$ with $n > 1$.
Let $B$ be the projection of $A$ onto its first $n-1$ coordinates,
and let $F : B \rightrightarrows \dR$ be the set-valued function whose graph is $A$; that is
\[ F(x) := \{ t \in \dR \colon (x,t) \in A\}, \ \ \ \forall x \in B. \]
Note that $F(x)$ is an open interval for every $x \in B$.
The set $B$ is open and monovex, and by the induction hypothesis it is contractible.
The set-valued function $F$ satisfies the conditions of Michael's selection theorem (Michael, 1956, Theorem 3.1'\kern-.06em'\kern-.06em'),%
\footnote{Michael's selection theorem implies in particular that for every subset $X \subseteq \dR^n$
and every set-valued function $F : X \rightrightarrows \dR^m$
with an open graph and nonempty convex values there exists a continuous function $f : X \to \dR^n$ such that $f(x) \in F(x)$ for every $x \in X$.}
hence there is a continuous function $f : B \to \dR$ such that $f(x) \in F(x)$ for every $x \in B$.
This implies that $A$ is contractible;
indeed first contract $A$ to $\graph(f)$, and then contract $\graph(f)$ to a point.
\end{proof}

\bigskip

The main result of the paper is the following.

\begin{theorem}
\label{theorem:closed}
Every closed monovex subset of $\dR^n$ is contractible.
\end{theorem}

We provide two proofs to Theorem \ref{theorem:closed},
each one uses different properties of monovex sets, which may have their own interest.
The first proof, provided in Section \ref{proof:2},
relies on the property that one can assign, in a continuous way, to every pair of points in a monovex set a (not necessarily monotone)
path that connects these points and lies in the set.
The second proof, provided in Section \ref{proof:1},
relies on the stronger property that the complement of a monovex set can be continuously projected onto the set.
In the first proof we will provide a direct argument that shows the existence of a continuous map from pairs of points in the monovex set to paths that connect the points and lie in the set.

Several open problems regarding contractibility of monovex sets still remain.
We prove that every closed monovex set is contractible.
We do not know whether for every such set there is a Lipschitz continuous contraction.
Another issue that remains open is whether our results hold for infinite dimensional spaces.

\section{Proof of Theorem \ref{theorem:closed}}

Throughout the paper we use the maximum metric in $\dR^n$, that is, $d_\infty(x,y) := \max_{1 \leq i \leq n}| x_i-y_i|$ for every $x,y \in \dR^n$.
The distance between a point $x \in \dR^n$ and a set $A \in \dR^n$ is $d_\infty(x,A) := \inf_{y \in A} d_\infty(x,y)$,
and the distance between two sets $A,B \subset \dR^n$ is the Hausdorff distance
$d_\infty(A,B) := \max\{\sup_{x \in A} d_\infty(x,B),\sup_{y \in B} d_\infty(y,A)\}$.

For every $x \in \dR^n$ and every $r > 0$ we denote by $B(x,r) := \{y \in \dR^n \colon d_\infty(x,y) < r\}$
the open ball around $x$ with radius $r$,
and by $\overline B(x,r) := \{y \in \dR^n \colon d_\infty(x,y) \leq r\}$
the closed ball around $x$ with radius $r$.
We denote by $\vec 0$ the vector $(0,0,\ldots,0)$ in $\dR^n$.

A (closed) \emph{box} in $\dR^n$ is a set of the form $\times_{i=1}^n [a_i,b_i]$, where $a_i \leq b_i$ for each $i \in \{1,2,\ldots,n\}$.
A box is \emph{$l$-dimensional} if the number of indices $i$ such that $a_i < b_i$ is $l$.
The set of vertices of a box $R$ is denoted $\vertices(R)$.
The smallest box that contains a set $A$ is called the \emph{b-hull} of $A$
and denoted $\phull(A)$.

A \emph{b-lattice} is a set of the form $\Gamma = \{(a_1k_1,\ldots,a_nk_n) \colon k_1,k_2,\ldots,k_n \in \dZ\}$,
where $a_1,a_2,\ldots,a_n > 0$.
Denote by $P_l(\Gamma)$ the set of $l$-dimensional elementary boxes having vertices in the lattice, that is,
the collection of all sets $\times_{i=1}^n J_i$ such that for each $i$ either $J_i = \{a_ik_i\}$ or $J_i=[a_ik_i,a_i(k_i+1)]$ for some $k_i \in \dZ$,
and moreover the second condition happens for exactly $l$ values of $i$.
Denote $P(\Gamma) := P_n(\Gamma)$ the set of full-dimensional elementary boxes with vertices in the lattice $\Gamma$.

\subsection{First Proof}
\label{proof:2}

The following lemma states that any function $f$ from an $m$-dimensional grid to a monovex set $A$ can be extended to a continuous function
from the $m$-dimensional space to $A$ with the property that the image under $f$ of any elementary $l$-dimensional box whose vertices are points in the grid
is a subset of the b-hull of the image under $f$ of the vertices of the box.

\begin{proposition}
\label{lemma:lev1}
Let $A \subset \dR^n$ be a closed monovex set and let $\Gamma \subset \dR^m$ be a b-lattice.
Let $X \subset \dR^m$ be a (finite or infinite) union of boxes in $P(\Gamma)$,
and let $S := X \cap \Gamma$ be the set of all vertices of these boxes.
Let $f : S \to A$.
Then $f$ can be extended to a continuous function $f : X \to A$ that satisfies the following property:
\begin{itemize}
\item[(P)]
For every $l$, $1 \leq l \leq m$,
and every box $R \in P_l(\Gamma)$ that is a subset of $X$, the image $f(R)$ is a subset of the b-hull of $f(\vertices(R))$.
\end{itemize}
\end{proposition}

\begin{proof}
Assume w.l.o.g.~that $\Gamma = \dZ^m$.
To prove the result, we will define the function $f$ iteratively on the sets
$X \cap \left((\tfrac{1}{2^{k+1}}\Gamma) \setminus (\tfrac{1}{2^k}\Gamma)\right)$, $k=0,1,2,\ldots$,
and show that this definition can be extended to a continuous function over $X$.

For every natural number $r \in \dN$ and every integer $i$, $1 \leq i \leq n$, define a function
$\varphi_{r,i} : A^r \to A$ as follows.
Let $(q^{(1)},q^{(2)},\ldots,q^{(r)}) \in A^r$,
let $j_{\min}$ be an index in which $\min_{1 \leq j \leq r} q^{(j)}_i$ is attained,
and let $j_{\max}$ be an index in which $\max_{1 \leq j \leq r} q^{(j)}_i$ is attained.
Choose a continuous monotone curve $\gamma : [0,1] \to A$ connecting $q^{(j_{\min})}$ and $q^{(j_{\max})}$
(if $j_{\min}=j_{\max}$, the curve is constant).
By continuity there exists $t_0 \in [0,1]$ such that $\gamma_i(t_0) = \tfrac{q^{(j_{\min})}_i + q^{(j_{\max})}_i}{2}$.
Set
\[ \varphi_{r,i}(q^{(1)},q^{(2)},\ldots,q^{(r)}) := \gamma(t_0). \]

We now extend the function $f$ from $X \cap \tfrac{1}{2^k}\Gamma$ to $X \cap \tfrac{1}{2^{k+1}}\Gamma$, for $k=0,1,2,\ldots$.
Suppose then that $f : X \cap \tfrac{1}{2^k}\Gamma \to A$ is given,
and set $i := k+1\ (\modulo n)$.
Every $q \in X \cap \left((\tfrac{1}{2^{k+1}}\Gamma) \setminus (\tfrac{1}{2^k}\Gamma)\right)$
is the center of a unique $l$-dimensional box $R \in P_l(\tfrac{1}{2^k}\Gamma)$
that is contained in $X$ (where $1 \leq l \leq n$).
Define
\[ f(q) := \varphi_{2^l,i}(q^{(1)},q^{(2)},\ldots,q^{(2^l)}), \]
where $q^{(1)},q^{(2)},\ldots,q^{(2^l)}$ are the vertices of $R$.
Note that this definition verifies Property (P).

The function $f$ is locally uniformly continuous,%
\footnote{A function is \emph{locally uniformly continuous} if it is uniformly continuous on every bounded subset.}
and in fact, locally $\tfrac{1}{n}$-H\"older.
Indeed, for $k \in \dN$ and a box $R \in P(\tfrac{1}{2^k}\Gamma)$ denote
$M_i(R) := \max\{ |f_i(q^{(j)}) - f_i(q^{(m)})| \colon q^{(j)},q^{(m)} \in \vertices(R)\}$.
Let $N_i(R) := \max\{ M_i(S) \colon S  \in P(\tfrac{1}{2^{k+1}}\Gamma), S \subset R\}$
be the maximum of the corresponding quantity over all sub-boxes of $S$ that belong to $P(\tfrac{1}{2^{k+1}}\Gamma)$.
If $i = k+1\ (\modulo n)$ then $N_i(R) \leq M_i(R)/2$, while if $i \neq k+1\ (\modulo n)$ then $N_i(R) \leq M_i(R)$.
Since $i= k+1\ (\modulo n)$ infinitely often with step $n$ as $k$ increases,
and since $f$ satisfies Property (P),
it follows that $f$ is indeed locally $\tfrac{1}{n}$-H\"older continuous.

This implies that $f$ can be extended to a continuous function
$f : X \cap \left( \cup_{k=0}^\infty \tfrac{1}{2^k}\Gamma\right) \to A$.
The set $\cup_{k=0}^\infty \tfrac{1}{2^k}\Gamma$ is dense in $X$, hence $f$ can be extended
to a continuous function from $X$ to $A$ that satisfies Property (P).
The extended function $f$ is locally $\tfrac{1}{n}$-H\"older continuous as well.
\end{proof}

\bigskip

We would like to prove that there is a continuous function $f : A \times A \times [0,1] \to A$ that satisfies $f(x,y,0) = x$ and $f(x,y,1) = y$
for every $x,y \in A$.
In the next lemma we prove an approximate version of this result.
We will use it in Proposition \ref{lemma:lev3} below to prove the stronger version of the claim.

\begin{lemma}
\label{lemma:lev2}
Let $A \subset \dR^n$ be a closed monovex set.
For every $\delta > 0$ there exists a continuous function $g_\delta : A \times A \times [0,1] \to A$ such that for every
$x,y \in A$ we have:
\begin{enumerate}
\item   $d_\infty(x,g_\delta(x,y,0)) \leq \delta$ and $d_\infty(y,g_\delta(x,y,1)) \leq \delta$.
\item   $d_\infty(g_\delta(x,y,t),\phull(\{x,y\})) \leq \delta$ for every $t \in [0,1]$.
\end{enumerate}
\end{lemma}

\begin{proof}
Fix $\delta > 0$. Consider the lattice $\Gamma := \tfrac{\delta}{2}\dZ^n$,
and denote by $X$ the union of all boxes $R \in P(\Gamma)$ that satisfy $R \cap A \neq \emptyset$.
Denote $\widetilde X := X \times X \times [0,1] \subseteq \dR^{2n+1}$,
$\widetilde\Gamma := \Gamma \times \Gamma \times \dZ$,
and $\widetilde S := \widetilde X \cap \widetilde \Gamma$.

Let $f : \widetilde S \to A$ be any function that satisfies the following property:
for every $x,y \in X \cap \Gamma$ we have
$d_\infty(x,f(x,y,0)) \leq \tfrac{\delta}{2}$
and $d_\infty(y,f(x,y,1)) \leq \tfrac{\delta}{2}$.
Such a function exists
since every $x \in X \cap\Gamma$ is a vertex of a box $R \in P(\Gamma)$ whose sidelength is $\tfrac{\delta}{2}$ with $R \cap A \neq\emptyset$.

By Proposition \ref{lemma:lev1}, the function $f$ can be extended to a continuous function $f : \widetilde X \to A$ that satisfies Property (P).
In particular, for every two boxes $Q,R \in P(\Gamma)$ lying in $X$, we have:
\begin{itemize}
\item
$f(Q \times R \times \{0\})$ is contained in the b-hull of $f(\vertices(Q) \times \vertices(R) \times \{0\})$;
\item
$f(Q \times R \times \{1\})$ is contained in the b-hull of $f(\vertices(Q) \times \vertices(R) \times \{1\})$.
\end{itemize}
Moreover, for every $x \in Q$, every $y \in R$, every $q \in \vertices(Q)$, and every $r \in \vertices(R)$ we have
$d_\infty(x,q) \leq \tfrac{\delta}{2}$,
$d_\infty(y,r) \leq \tfrac{\delta}{2}$,
$d_\infty(q,f(q,r,0)) \leq \tfrac{\delta}{2}$,
and
$d_\infty(r,f(q,r,1)) \leq \tfrac{\delta}{2}$.
By the triangle inequality it follows that
$d_\infty(x,f(q,r,0)) \leq \delta$
and
$d_\infty(y,f(q,r,1)) \leq \delta$.
We conclude that given $x \in Q$ and $y \in R$,
for every $q \in \vertices(Q)$ and $r \in \vertices(R)$ we have
$d_\infty(x,f(q,r,0)) \leq \delta$
and
$d_\infty(y,f(q,r,1)) \leq \delta$.
Therefore, since $f(x,y,0) \in \phull(f(\vertices(Q) \times \vertices(R) \times \{0\}))$ and
$f(x,y,1) \in \phull(f(\vertices(Q) \times \vertices(R) \times \{1\}))$,
we have
$d_\infty(x,f(x,y,0)) \leq \delta$
and
$d_\infty(y,f(x,y,1)) \leq \delta$.

In addition, since $f$ satisfies Property (P),
the image $f(Q \times R \times [0,1])$ is contained in the b-hull of
$f(\vertices(Q) \times \vertices(R) \times \{0,1\})$,
and hence we also conclude that
$d_\infty(f(x,y,t),\phull(\{x,y\})) \leq \delta$ for every $t \in [0,1]$.

To summarize, for every $x,y \in X$ and every $t \in [0,1]$ we have
\begin{itemize}
\item   $d_\infty(x,f(x,y,0)) \leq \delta$,
\item   $d_\infty(y,f(x,y,1)) \leq \delta$, and
\item   $d_\infty(f(x,y,t),\phull(\{x,y\})) \leq \delta$.
\end{itemize}
To end the proof of the lemma, define $g_\delta$ to be the restriction of $f$ to $A \times A \times [0,1]$.
\end{proof}

\begin{proposition}
\label{lemma:lev3}
There exists a continuous function $\varphi : A \times A \times [0,1]\to A$ such that
$\varphi(x,y,0) = x$ and $\varphi(x,y,1) = y$ for every $x,y \in A$.
\end{proposition}

We note that Proposition \ref{lemma:lev3} implies Theorem \ref{theorem:closed}.
Indeed, choose an arbitrary $x_0 \in A$.
The function $G : A \times [0,1] \to A$ defined by $G(x,t) := \varphi(x,x_0,t)$ for every $x \in A$ and $t \in [0,1]$
is a homotopy between $A$ and $\{x_0\}$.

\bigskip

\begin{proof}[Proof of Proposition \ref{lemma:lev3}]
Let $(\delta_k)_{k=1}^\infty$ be a sequence of positive reals such that $\sum_{k=1}^\infty \delta_k < \infty$.
We define the function $\varphi$ in steps by a Cantor set construction.
Define $C_0 := \{0,1\}$, $C_1 := \{ [\tfrac{1}{3},\tfrac{2}{3}] \}$, and for every $k \geq 2$ let $C_k$
be the collection of all closed intervals $[s,s']$ where $s = \tfrac{1}{3^k} + \sum_{j=1}^{k-1}\tfrac{\alpha_j}{3^j}$
and $s' = \tfrac{2}{3^k} + \sum_{j=1}^{k-1}\tfrac{\alpha_j}{3^j}$, for some $\alpha_j \in \{0,2\}$, $j=1,2,\ldots,k-1$ (see Figure \arabic{figurecounter}).

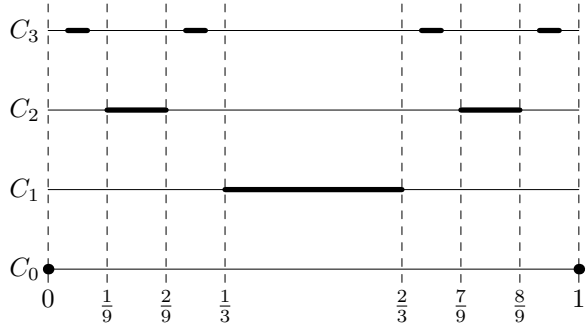
\bigskip
\centerline{\includegraphics{figure.4}}
\centerline{Figure \arabic{figurecounter}: The sets $C_k$ for $k=0,1,2,3$.}
\addtocounter{figurecounter}{1}
\bigskip

For $k = 0,1,\ldots$, in step $k$ we define $\varphi$ on $A \times A \times \left(\cup_{[s,s'] \in C_k} [s,s']\right)$.
For $k=0$ set
\[ \varphi(x,y,0) := x, \ \ \ \varphi(x,y,1) := y, \ \ \ \forall x,y \in A. \]
For $k \geq 1$, consider an interval $[s,s'] \in C_k$ and set $t := s - \tfrac{1}{3^k}$ and $t' := s' + \tfrac{1}{3^k}$.
If $k=1$, the points $t$ and $t'$ are 0 and 1, which lie in $C_0$.
If $k \geq 2$, one of these points is an endpoint of an interval in $C_{k-1}$ and
the other is an endpoint of an interval in $C_{k-2}$.
In both cases $\varphi(\cdot,\cdot,t)$ and $\varphi(\cdot,\cdot,t')$ were already defined.
Set
\[ \varphi(x,y,(1-\lambda)s+\lambda s') := g_{\delta_k}(\varphi(x,y,t),\varphi(x,y,t'),\lambda), \ \ \ \forall x,y \in A, \forall \lambda \in [0,1], \]
where $g_{\delta_k}$ satisfies the statement of Lemma \ref{lemma:lev2}.
The procedure described above defines $\varphi$ on $A \times A \times \left(\cup_{k=0}^\infty \cup_{[s,s'] \in C_k} [s,s']\right)$.
The set $\cup_{k=0}^\infty \left(\cup_{[s,s'] \in C_k} [s,s']\right)$ is dense on $[0,1]$, hence $\varphi$ is defined in a dense subset of $A \times A \times [0,1]$.
Since $\sum_{k=1}^\infty \delta_k < \infty$, the function $\varphi$ is in fact locally uniformly continuous,
hence it can be extended to a continuous function $\varphi : A \times A \times [0,1] \to A$,
as desired.
\end{proof}

\subsection{Second Proof}
\label{proof:1}

We first argue that if $A$ is a monovex set
and $R$ is an open box whose faces are parallel to the axes,
then $A+R$ is monovex.
We note that the proof is valid also when the box $R$ is closed.

\begin{lemma}
\label{lemma:0}
If the set $A \subset \dR^n$ is monovex and $R \subset \dR^n$ is an open box whose faces are parallel to the axes,
then the set $A + R$ is monovex.
\end{lemma}

\begin{proof}
Let $x,y \in A + R$.
Then $x=x'+a'$ and $y=y'+b'$, where $x',y' \in A$ and $a',b' \in R$.
Assume w.l.o.g.~that $x'_i \leq y'_i$ for every $i \in \{1,2,\ldots,n\}$,
and let $\gamma' : [0,1] \to A$ be a continuous monotone path that connects $x'$ to $y'$.
Let $J := \{ i \colon 1 \leq i \leq n, x'_i < y'_i\}$.
There are a diagonal matrix $D \in \mathcal{M}_{n,n}(\dR)$ and a vector $v \in \dR^n$ such that
$(D x' + v)_i = a'_i$ and $(D y' + v)_i = b'_i$ for every coordinate $i \in J$.
Define for every coordinate $i \in \{1,2,\ldots,n\}$ a continuous function $\delta_i \colon [0,1] \to \dR$ as follows:
\begin{itemize}
\item   If $i \in J$ then $\delta_i(t) := (D\gamma'(t) + v)_i$.
\item   If $i \not\in J$ then $\delta_i$ is any continuous monotone function that satisfies $\delta_i(0) = a'_i$ and
$\delta_i(1) = b'_i$.
\end{itemize}
Since $\delta_i$ is monotone for every coordinate $i$ and since $\delta(0) = a'$ and $\delta(1) = b'$,
we have $\delta(t) \in R$ for every $t \in [0,1]$.

The path $\gamma := \gamma' + \delta$ satisfies the following properties,
which imply that $\gamma$ is a continuous monotone path in $A+R$ from $x$ to $y$.
\begin{itemize}
\item $\gamma(0) = x' + a' = x$ and $\gamma(1) = y'+b' = y$.
\item   $\gamma(t) \in A + R$ for every $t \in [0,1]$.
\item
For every coordinate $i \in J$ we have $\gamma_i = ((I + D)\gamma' + v)_i$.
Since $I+D$ is a diagonal matrix, the function $\gamma_i$ is monotone.
\item
For every coordinate $i \not\in J$ we have $\gamma_i = x'_i + \delta_i$,
and therefore in this case $\gamma_i$ is monotone as well.
\end{itemize}
Since $x$ and $y$ are arbitrary, the result follows.
\end{proof}

\bigskip
We will use the following extension of Michael's selection theorem to monovex-valued functions.

\begin{lemma}
\label{lemma:michael}
Let $X \subseteq\dR^n$ and let $F : X \rightrightarrows \dR^m$ be a set-valued function with open graph and nonempty monovex values.
Then $F$ has a continuous selection: there is a continuous function $f : X \to \dR^m$ that satisfies $f(x) \in F(x)$ for every $x \in X$.
\end{lemma}

\begin{proof}
We prove the result by induction on $m$.
If $m=1$ then the values of $F$ are convex, hence by Michael's selection theorem $F$ has a continuous selection $f$.

Assume now that $m > 1$.
Let $F_1 : X \rightrightarrows \dR$ be the projection of $F$ to its first coordinate:
\[ F_1(x) = \{y_1 \in \dR \colon (y_1,y_2,\ldots,y_m) \in F(x) \hbox{ for some } (y_2,\ldots,y_m) \in \dR^{m-1}\}. \]
Let $F_2 : \graph(F_1) \rightrightarrows \dR^{m-1}$ be the set-valued function defined by
\[ F_2(x,y_1) := \{ (y_2,\ldots,y_m) \in \dR^{m-1} \colon (y_1,y_2,\ldots,y_m) \in F(x)\}. \]
The set-valued functions $F_1$ and $F_2$ have open graphs and monovex values, hence by the induction hypothesis applied to both of them
there are continuous selections $f_1$ of $F_1$ and $f_2$ of $F_2$.
The function $g : X \to \dR^m$ defined by $g(x) := (f_1(x),f_2(x,f_1(x)))$ is a continuous selection of $F$.
\end{proof}

\begin{definition}
Let $U \subseteq \dR^n$ be an open set, let $\ep : U \to (0,1]$ be a continuous function,
and let $F : U \rightrightarrows \dR^m$ be a set-valued function.
The \emph{$\ep$-neighborhood} of $F$ is the set
\[ \calN_\ep(F) := \bigcup_{(x,y) \in \graph(F)} B((x,y),\ep(x)) \subseteq \dR^{n+m}. \]
\end{definition}

The following result states that every set-valued function with a relatively closed graph
and compact monovex values can be approximated by a set-valued function with an open graph and monovex values.

\begin{lemma}
\label{proposition:1}
Let $U \subseteq \dR^n$ be an open set, let $\ep : U \to (0,1]$ be a continuous function,
and let $F : U \rightrightarrows \dR^m$ be a set-valued function with a relatively closed graph and compact monovex values.
There exists a set-valued function $G : U \rightrightarrows \dR^m$ with an open graph and monovex values satisfying
$\graph(F) \subseteq \graph(G) \subseteq \calN_\ep(F)$.
\end{lemma}

\begin{proof}

\noindent\textbf{Step 1:} Definitions.

Define a function $\eta : U \to \{\tfrac{1}{2^k} \colon k \in \dN\}$ by
\[ \eta(x) := \max\left\{ \tfrac{1}{2^k} \colon k \in \dN, \tfrac{1}{2^k} \leq \tfrac{\ep(x)}{10}\right\}. \]
This function is upper-semi-continuous function:
for every sequence $(x_k)_{k \in \dN} \subset U$ that converges to a limit $x \in U$ we have
$\limsup_{k \to \infty} \eta(x_k) \leq \eta(x)$.
Given $\delta > 0$,
let $\calG_\delta := P(\delta\dZ^m)$ be the collection of elementary $m$-dimensional boxes in the lattice $\delta\dZ^m$.
Let $F_1(x)$ be the union of all boxes in $\calG_{\eta(x)}$ that have nonempty intersection with $F(x)$:
\[ F_1(x) := \bigcup\{ R \in \calG_{\eta(x)} \colon R \cap F(x) \neq \emptyset\}. \]
The set $F_1(x)$ contains $F(x)$,
it is a uniion of closed boxes, hence closed, and it approximates $F(x)$:
$d_\infty(F_1(x),F(x)) \leq \eta(x) \leq \tfrac{\ep(x)}{10}$ for every $x \in U$.

\noindent\textbf{Step 2:} The set $F_1(x)$ is monovex for every $x \in U$.

Let $x \in U$ and let $y,z \in F_1(x)$.
By the definition of $F_1$, there are $y',z' \in F(x)$ and two boxes $R,S \in \calG_{\eta(x)}$ such that
$y,y' \in R$ and $z,z' \in S$.
Since $F(x)$ is monovex, there is a continuous monotone path $\gamma'$ that connects $y'$ to $z'$ within $F(x)$.
Assume w.l.o.g.~that $y'_i \leq z'_i$ for every $i=1,2,\ldots,n$.

We now define a path $\gamma$.

\begin{enumerate}
\item[(B1)]
If there is $a_i \in \dZ$ such that $a_i \eta(x) \leq z_i, y_i \leq (a_i+1)\eta(x)$, set $\gamma_i(t) := (1-t)y_i + tz_i$.
\item[(B2)]
Otherwise there is $a_i \in \dZ$ such that $y_i \leq a_i\eta(x) \leq z_i$.
We let $\gamma_i(t)$ be the projection of $\gamma'_i(t)$ to the line segment $[y_i,z_i]$:
\[ \gamma_i(t) := \min\{ \max\{ \gamma'_i(t),y_i\}, z_i\}. \]
\end{enumerate}

The reader can verify that $\gamma$ is contained in $F_1(x)$.
However, $\gamma(0)$ need not be $y$ and $\gamma(1)$ need not be $z$.
Indeed, for every $i$ for which Condition (B2) holds we have $\gamma_i(0) =\max\{y_i,y'_i\}$
and $\gamma_i(1) =\min\{z_i,z'_i\}$.
Define then two points $\widetilde y,\widetilde z \in F_1(x)$ by
\begin{eqnarray}
\widetilde y_i &:=& \left\{
\begin{array}{lll}
y_i & \ \ \ \ \ & \hbox{Condition (B1) holds,}\\
\max\{y_i,y'_i\} & & \hbox{Condition (B2) holds.}
\end{array}
\right.\\
\widetilde z_i &:=& \left\{
\begin{array}{lll}
z_i & \ \ \ \ \ & \hbox{Condition (B1) holds,}\\
\min\{z_i,z'_i\} & & \hbox{Condition (B2) holds.}
\end{array}
\right.
\end{eqnarray}
A monotone path in $F_1(x)$ that connects $y$ and $z$
is the concatenation of (a) a monotone path that connects $y$ and $\widetilde y$,
(b) the path $\gamma$, and
(c) a monotone path that connects $\widetilde z$ to $z$.

\noindent\textbf{Step 3:}
For every $x \in U$ there is $\delta_x \in (0,\tfrac{\ep(x)}{10})$ such that
$F_1(y) \subseteq F_1(x)$ and $\eta(y) \leq \eta(x)$ for every $y \in B(x,\delta_x)$.

Since the function $\eta$ is upper-semi-continuous and its image is discrete, for every $x\in U$ there is $\delta_x > 0$ such that $\eta(y) \leq \eta(x)$,
for every $y \in B(x,\delta_x)$.
We turn to prove the analogous property for $F_1$.
If the property does not hold, then for every $k \in \dN$ there exists $y_k \in B(x,\tfrac{1}{k})$
such that $F_1(y_k) \not\subseteq F_1(x)$.
That is, there is $z_k \in F_1(y_k) \setminus F_1(x)$.
Since $z_k \in F_1(y_k)$, the point $z_k$ belongs to some box $R_k$ of the lattice $\calG_{\eta(y_k)}$,
and in particular there is a point $w_k \in R_k \cap F(y_k)$.
Since (i) $F$ has compact values, (ii) the image of $\eta$ is discrete,
and (iii) $\eta$ is locally bounded from below,
it follows that the number of boxes $R_k$ that satisfy these properties is finite, hence by taking a subsequence we can assume that
(a) $R_k = R$ for every $k \in \dN$ and
(b) the sequence $(w_k)_{k \in \dN}$ converges to some point $w \in \dR^n$.
In particular, $w \in R$.
Since the graph of $F$ is relatively closed, $w \in F(x) \cap R$.
In particular, $F(x) \cap R \neq \emptyset$, and hence $R \subseteq F_1(x)$, which implies that $z_k \in F_1(x)$ for every $k \in \dN$, a contradiction.

\noindent\textbf{Step 4:} Definition of the set-valued function $G$.

For every $x \in U$ define a set $Q(x)$ by
\[ Q(x) := \{ y \in U \colon x \in B(y,\tfrac{\delta_y}{2})\}. \]
Thus, $y \in Q(x)$ if the two points $x$ and $y$ are close, when the distance is measured by $\delta_y$.
Note that $x \in Q(x)$ for every $x \in U$, and therefore $Q$ has nonempty values.
Define
\[ G(x) := \bigcup_{ y \in Q(x)} \bigl( F_1(y)  + B(\vec 0,\eta(y)) \bigr), \ \ \ \forall x \in U. \]
We will prove that the set-valued function $G$ satisfies the desired conditions.

Note that $G(x)$ is a union of open sets, and hence it is open.
In addition, since $x \in Q(x)$, we have $G(x) \supseteq F_1(x) \supseteq F(x)$,
hence $\graph(G) \supseteq \graph(F)$.

\noindent\textbf{Step 5:}
If $y_1,y_2 \in Q(x)$ then either
(a) $F_1(y_1) \subseteq F_1(y_2)$ and $\eta(y_1) \leq \eta(y_2)$, or
(b) $F_1(y_1) \supseteq F_1(y_2)$ and $\eta(y_1) \geq \eta(y_2)$.

Let $y_1,y_2 \in Q(x)$
and assume w.l.o.g.~that $\delta_{y_1} \geq \delta_{y_2}$.
Since $x \in B(y_1,\tfrac{\delta_{y_1}}{2}) \cap B(y_2,\tfrac{\delta_{y_2}}{2})$
we deduce that $B(y_1,\tfrac{\delta_{y_1}}{2}) \cap B(y_2,\tfrac{\delta_{y_2}}{2}) \neq \emptyset$.
In particular
\[ d_\infty(y_1,y_2) < \tfrac{\delta_{y_1}}{2} +\tfrac{\delta_{y_2}}{2} \leq \delta_{y_1}, \]
which implies that $y_2 \in B(y_1, \delta_{y_1})$.
By Step 3 this implies that $F_1(y_1) \supseteq F_1(y_2)$ and $\eta(y_1) \geq \eta(y_2)$.

\noindent\textbf{Step 6:}
The set $G(x)$ is monovex for every $x \in U$.

Let $z_1,z_2 \in G(x)$.
Then there are $y_1,y_2 \in Q(x)$ such that
$z_1 \in F_1(y_1) + B(\vec 0,\eta(y_1))$ and $z_2 \in F_1(y_2) + B(\vec 0,\eta(y_2))$.
By Step~5 we can assume w.l.o.g.~that $z_2 \in F_1(y_1) + B(\vec 0,\eta(y_1))$.
By Step~2 and Lemma~\ref{lemma:0} the set $F_1(y_1) + B(\vec 0,\eta(y_1))$ is monovex.

\noindent\textbf{Step 7:}
The graph of $G$ is an open subset of the $\tfrac{3}{10}\ep$-neighborhood of $F$.

By the definition of $G$,
\[ \graph(G) = \left(\bigcup_{y \in U} B(y,\tfrac{\delta_y}{2}) \times (F_1(y) + B(\vec 0,\eta(y)))\right) \cap (U \times \dR^m). \]
It follows that $\graph(G)$ is a union of open sets, hence open.
Moreover, $\graph(G)$ is a subset of the $(\tfrac{\delta_y}{2}+2\eta(y))$-neighborhood of $F$.
The claim follows since $(\tfrac{\delta_y}{2}+2\eta(y))\leq \tfrac{3}{10}\ep$.
\end{proof}

\bigskip

The following result implies Theorem \ref{theorem:closed}.

\begin{proposition}
\label{prop:4}
Every closed monovex set $A \subseteq \dR^n$ is a retract:
there is a continuous function $h : \dR^n \to A$ which is the identity on $A$.
\end{proposition}

We now show that Proposition \ref{prop:4} implies Theorem \ref{theorem:closed}.
Indeed, fix $x_0 \in \dR^n$, and let $h$ be the retract of Proposition \ref{prop:4}.
The function $h^* : A \to [0,1] \to A$ defined by
\[ h^*(x,t) := h((1-t)x+tx_0) \]
is a homotopy between $A$ and $\{h(x_0)\}$,
and therefore $A$ is contractible, as claimed.

\bigskip

\begin{proof}[Proof of Proposition \ref{prop:4}]

\noindent\textbf{Step 1:} Definitions.

The real-valued function $x \mapsto d_\infty(x,A)$ is continuous and positive for $x \in \dR^n \setminus A$.
Let $F : \dR^n \setminus A \rightrightarrows A$ be the set-valued function defined by
\[ F(x) := A \cap \overline{B}(x,d_\infty(x,A)). \]
The set $F(x)$ contains all points in $A$ that are closest to $x$.
Since $A$ is a closed monovex set, the set-valued function $F$ has a relatively closed graph and compact monovex values.

Set $\ep(x) := \tfrac{d_\infty(x,A)}{10}$ and apply Lemma \ref{proposition:1} to $F$ and $\ep$.
It follows that there exists a set-valued function $G : \dR^n \setminus A \to \dR^n$ with open graph and monovex values such that
$\graph(G)$ lies in the $\ep$-neighborhood of $F$.

\noindent\textbf{Step 2:}
For every $x \in \dR^n \setminus A$ we have $G(x) \subseteq B(x,\tfrac{4}{3}d_\infty(x,A))$.

Let $(x,z) \in \graph(G)$.
Since $\graph(G)$ is contained in an $\ep$-neighborhood of $F$,
there is $(x',z') \in \graph(F)$ such that $d_\infty(x',x) < \ep(x')$
and $d_\infty(z', z) < \ep(x')$.
Since $z' \in F(x')= A \cap \overline B(x',d_\infty(x',A))$,
it follows that $d_\infty(z',x') = d_\infty(x',A)$.
By the triangle inequality
\begin{eqnarray}
\label{equ:4}
d_\infty(x,A) &\geq& d_\infty(x',A) - d_\infty(x,x') > d_\infty(x',A) - \ep(x')\\
&=& d_\infty(x',A) - \tfrac{d_\infty(x',A)}{10} = \tfrac{9}{10}d_\infty(x',A).
\nonumber
\end{eqnarray}
By the triangle inequality once again and (\ref{equ:4}) we obtain that
\begin{eqnarray}
d_\infty(z,x) &\leq& d_\infty(z,z') + d_\infty(z',x') + d_\infty(x',x)\\
&<& 2\ep(x') + d_\infty(x',A) = \tfrac{12}{10}d_\infty(x',A) \leq \tfrac{12}{9}d_\infty(x,A),
\end{eqnarray}
as claimed.

\noindent\textbf{Step 3:} Definition of a function $g$.

By Michael's selection theorem (Lemma \ref{lemma:michael}), there is a continuous selection $g$ of $G$.
Step 2 implies that
\begin{equation}
\label{equ:10}
d_\infty(g(x),x) \leq \tfrac{4}{3} d_\infty(x,A), \ \ \ \forall x \in \dR^n \setminus A.
\end{equation}
As a consequence we obtain that for every sequence $(x_k)_{k \in \dN}$ that converges to a limit $x$ that lies in $A$ we have
$\lim_{k \to \infty} d_\infty(g(x_k),x_k) = 0$.
In particular, the function $g$ can be extended to a continuous function from $\dR^n$ to $\dR^n$ that is the identity on $A$.

Since $(x,g(x)) \in \graph(G)$
and since $G$ lies in an $\ep$-neighborhood of $F$,
it follows that there is $(x',y') \in \graph(F)$ such that $d_\infty(x,x') < \ep(x')$ and $d_\infty(g(x),y') < \ep(x')$.
Since $(x',y') \in \graph(F)$ we in particular deduce that $y' \in A$, so that
\[ d_\infty(g(x),A) < \ep(x') = \tfrac{d_\infty(x',A)}{10} \leq \tfrac{d_\infty(x,A)}{9}, \]
where the last inequality follows from Eq.~(\ref{equ:4}).

\noindent\textbf{Step 4:} Definition of the function $h$.

For every $k \in \dN$ let $g^k$ be the composition of $g$ on itself $k$ times;
that is, $g^1 := g$ and $g^k := g\circ g^{k-1}$.
For every $k \in \dN$ the function $g^k$ is the identity on $A$ and satisfies
$d(g^k(x),A) \leq \tfrac{d_{\infty}(x,A)}{9^k}$ for every $x \in \dR^n \setminus A$.
Together with (\ref{equ:10}) we deduce that the functions $(g^k)_{k \in \dN}$ converge locally uniformly to some continuous function $h$,
which is the desirable retract.
\end{proof}

\end{document}